\documentclass[11pt,a4paper]{article}

\usepackage{tikz,subfigure}
\usetikzlibrary{shapes.callouts,decorations.pathmorphing}
\usepackage[active]{srcltx}
\usepackage[margin=1in]{geometry}
\usepackage{amsthm}
\usepackage{epsfig}

\usepackage{epsfig,amssymb,amsfonts,amsmath}
\usepackage{multirow}
\usepackage{xspace}
\usepackage{tabularx}
\usepackage{xcolor}

\newcommand{\remove}[1]{}

\newcommand{\ce}{\mathrm{e}}

\renewcommand{\deg}{\operatorname{deg}}

\renewcommand{\leq}{\leqslant}
\renewcommand{\geq}{\geqslant}

\newcommand{\Oh}{\mathcal{O}}

\newcommand{\Hmax}[2]{\mathsf{H}_{\max}}

\newcommand{\cov}{\mathrm{cov}}

\newcommand{\lemref}[1]{Lemma~\ref{lem:#1}}

\newcommand{\Pro}[1]{\mathbf{Pr} \left[\,#1\,\right]}

\newcommand{\Mid}{\,\middle\vert\,}

\newcommand{\Ex}[1]{\mathbf{E} \left[\,#1\,\right]}

\newcommand{\ind}[1]{\mathbf{1}\left(#1\right)}

\renewcommand{\tilde}{\widetilde}
\renewcommand{\hat}{\widehat}
\renewcommand{\bar}{\overline}
\renewcommand{\epsilon}{\varepsilon}

\usepackage{setspace}

\newtheorem{thm}{Theorem}
\newtheorem{lem}[thm]{Lemma}
\newtheorem{rem}[thm]{Remark}
\newtheorem{pro}[thm]{Proposition}

\renewcommand{\tilde}{\widetilde}

\numberwithin{thm}{section}
\numberwithin{equation}{section}

\begin{document}

\title{Balls into bins via local search: cover time and maximum load\footnote{A preliminary version of this paper appeared in the 31st Symposium on Theoretical Aspects of Computer Science (STACS'14).}}

\author{
Karl Bringmann\thanks{Max Planck Institute for Informatics, Saarbr\"ucken, Germany.\ \ Email:
\hbox{karl.bringmann@mpi-inf.mpg.de}} \and
Thomas Sauerwald\thanks{Computer Laboratory, University of Cambridge, UK. \ \ Email: \hbox{thomas.sauerwald@cl.cam.ac.uk}} \and
Alexandre Stauffer\thanks{Department of Mathematical Sciences, University of Bath, UK.\ \ Email: \hbox{a.stauffer@bath.ac.uk}.
Supported in part by a Marie Curie Career Integration Grant PCIG13-GA-2013-618588 DSRELIS.} \and
He Sun\thanks{Max Planck Institute for Informatics, Saarbr\"ucken, Germany.\ \ Email:
\hbox{hsun@mpi-inf.mpg.de}}
}
\date{}

\maketitle

\begin{abstract} 
We study a natural process for allocating $m$ balls into $n$ bins that are organized as the vertices of an undirected graph $G$.
Balls arrive one at a time. When a ball arrives, it first chooses a vertex $u$ in $G$ uniformly at random.
Then the ball performs a local search in $G$ starting from $u$ until it reaches a vertex with local minimum load, where the ball is finally placed on.
Then the next ball arrives and this procedure is repeated.
For the case $m=n$, we give an upper bound for the maximum load on graphs with bounded degrees.
We also propose the study of the \emph{cover time} of this process, which is defined as the smallest $m$ so that every bin has at least one ball allocated to
it. We establish an upper bound for the cover time on graphs with bounded degrees. Our bounds for the maximum load and the cover time are tight when the
graph is vertex transitive or sufficiently homogeneous. We also give upper bounds for the maximum load when $m\geq n$.
\newline
\newline
\emph{Keywords and phrases.} balls-into-bins, load balancing, stochastic process, local search.
\end{abstract}


\section{Introduction}
A very simple procedure for allocating $m$ balls into $n$ bins is to place each ball into a bin chosen independently
and uniformly at random. We refer to this process as
\emph{1-choice process}. It is well known that, when $m=n$, the maximum load
for the 1-choice process (i.e., the maximum number of balls allocated to any single bin) is
$\Theta\left(\frac{\log n}{\log\log n}\right)$~\cite{RS98}.
Alternatively, in the $d$-choice process, balls arrive sequentially one after the other, and
when a ball arrives, it chooses $d$ bins independently and uniformly at random,
and places itself in the bin that currently has the smallest load among the $d$ bins (ties are broken uniformly at random).
It was shown by Azar et al.~\cite{ABKU99} and Karp et al.~\cite{KLH96} that the maximum load for the $d$-choice process with $m=n$ and $d\geq 2$
is $\Theta\left(\frac{\log\log n}{\log d}\right)$. The constants omitted in the $\Theta$ are known and, as shown by V\"ocking~\cite{V03},
they can be reduced with a slight modification of the $d$-choice process.
Berenbrink et al.~\cite{BCSV06} extended these results to the case $m\gg n$.

In some applications, it is important to allow each ball to choose bins in a \emph{correlated} way.
For example, such correlations occur naturally in distributed systems, where the bins represent
processors that are interconnected as a graph and the balls represent tasks that need to be
assigned to processors. From a practical point of view, letting each task choose $d$ independent
random bins may be undesirable, since the cost of accessing two bins which are far away in the
graph may be higher than accessing two bins which are nearby. Furthermore, in some contexts,
tasks are actually created by the processors, which are then able to forward tasks to other
processors to achieve a more balanced load distribution. In such settings, allocating balls close
to the processor that created them is certainly very desirable as it reduces the costs of probing the
load of a processor and allocating the task.

With this motivation in mind, Bogdan et al.~\cite{BSSS13} introduced a natural allocation process called
\emph{local search allocation}. Consider that the bins are organized as the vertices of a graph $G=(V,E)$ with $n=|V|$.
At each step a ball is ``born'' at a vertex chosen independently and uniformly at random from $V$, which we call the birthplace of the ball.
Then, starting from its birthplace, the ball performs a local search in $G$, where
in each step the ball moves to the adjacent vertex with the smallest load, provided that the load
is strictly smaller than the load of the vertex the ball is currently in.
We assume that ties are broken independently and uniformly at random.
The local search ends when the ball visits
the first vertex that is a local minimum, which is a vertex for which no neighbor has a smaller
load. After that, the next ball is born and the procedure above is repeated. See Figure~\ref{fig:localsearch} for an illustration.

\tikzstyle{knoten}=[circle, color=black, inner sep=2pt, fill=black]
\tikzstyle{edge}=[line width=0.5pt]

\tikzstyle{knoten}=[circle, color=black, inner sep=2pt, fill=black]
\tikzstyle{edge}=[line width=0.5pt]

\newcommand{\ball}[1]{ \shade[ball color=yellow, opacity=0.5] (#1) circle (.25cm) }
\newcommand{\nball}[1]{ \shade[ball color=red] (#1) circle (.25cm) }
\newcommand{\gball}[1]{ \draw [color=black, line width=0.7pt, fill=white,dashed, dash pattern=on 1pt off 1pt] (#1) circle (.22cm) }

\begin{figure}[t]

\begin{tikzpicture}[auto, scale=0.6,  knoten/.style={
           draw=black, fill=black, thin, circle, inner sep=0.05cm}]
           \draw[white] (0.8,2) -- (7,5);
           \node[knoten] (1) at (1,2) [label=below:$1$]{};
           \node[knoten] (2) at (2,2) [label=below:$2$]{};
           \node[knoten] (3) at (3,2) [label=below:$3$]{};
           \node[knoten] (4) at (4,2) [label=below:$4$]{};
           \node[knoten] (5) at (5,2) [label=below:$5$]{};
           \node[knoten] (6) at (6,2) [label=below:$6$]{};
           \node[] (7) at (3.5,1) [label=below:$(a)$]{};
           \draw[edge] (1) -- (2) -- (3) -- (4) -- (5) -- (6);
           \draw [color=black] (1)++(0.33,2.85) to ++(0,-2.6)  to ++(-0.7,0) to ++ (0,2.6);
           \draw [color=black] (2)++(0.33,2.85) to ++(0,-2.6)  to ++(-0.7,0) to ++ (0,2.6);
           \draw [color=black] (3)++(0.33,2.85) to ++(0,-2.6)  to ++(-0.7,0) to ++ (0,2.6);
           \draw [color=black] (4)++(0.33,2.85) to ++(0,-2.6)  to ++(-0.7,0) to ++ (0,2.6);
           \draw [color=black] (5)++(0.33,2.85) to ++(0,-2.6)  to ++(-0.7,0) to ++ (0,2.6);
           \draw [color=black] (6)++(0.33,2.85) to ++(0,-2.6)  to ++(-0.7,0) to ++ (0,2.6);

           \ball{1,2.6};
           \ball{2,2.6};
           \ball{3,2.6};
           \ball{3,3.1};
           \ball{4,2.6};
           \ball{4,3.1};
           \ball{4,3.6};
           \ball{5,2.6};
           \ball{5,3.1};
           \ball{6,2.6};
           \ball{6,3.1};
           \ball{6,3.6};
\end{tikzpicture}
\begin{tikzpicture}[auto, scale=0.6,  knoten/.style={
           draw=black, fill=black, thin, circle, inner sep=0.05cm}]
                     \draw[white] (0.5,2) -- (7,5);
           \node[knoten] (1) at (1,2) [label=below:$1$]{};
           \node[knoten] (2) at (2,2) [label=below:$2$]{};
           \node[knoten] (3) at (3,2) [label=below:$3$]{};
           \node[knoten] (4) at (4,2) [label=below:$4$]{};
           \node[knoten] (5) at (5,2) [label=below:$5$]{};
           \node[knoten] (6) at (6,2) [label=below:$6$]{};
               \node[] (7) at (3.5,1) [label=below:$(b)$]{};
           \draw[edge] (1) -- (2) -- (3) -- (4) -- (5) -- (6);
           \draw [color=black] (1)++(0.33,2.85) to ++(0,-2.6)  to ++(-0.7,0) to ++ (0,2.6);
           \draw [color=black] (2)++(0.33,2.85) to ++(0,-2.6)  to ++(-0.7,0) to ++ (0,2.6);
           \draw [color=black] (3)++(0.33,2.85) to ++(0,-2.6)  to ++(-0.7,0) to ++ (0,2.6);
           \draw [color=black] (4)++(0.33,2.85) to ++(0,-2.6)  to ++(-0.7,0) to ++ (0,2.6);
           \draw [color=black] (5)++(0.33,2.85) to ++(0,-2.6)  to ++(-0.7,0) to ++ (0,2.6);
           \draw [color=black] (6)++(0.33,2.85) to ++(0,-2.6)  to ++(-0.7,0) to ++ (0,2.6);
           \draw[dashed, line width=0.7pt, ->,  dash pattern=on 1pt off 1pt] (4,4.1) -- (3.2,3.7);
           \draw[dashed, line width=0.7pt, ->,  dash pattern=on 1pt off 1pt] (3,3.6) -- (2.2,3.2);
           \draw[line width=1pt, ->] (4,5.5) -- (4,4.4);
           \node[] (7) at (4,6.7) [label=below:ball $i$]{};
           \ball{1,2.6};
           \ball{2,2.6};
           \nball{2,3.1};
           \ball{3,2.6};
           \ball{3,3.1};
           \gball{3,3.6};
           \ball{4,2.6};
           \ball{4,3.1};
           \ball{4,3.6};
           \gball{4,4.1};
           \ball{5,2.6};
           \ball{5,3.1};
           \ball{6,2.6};
           \ball{6,3.1};
           \ball{6,3.6};
\end{tikzpicture}
\begin{tikzpicture}[auto, scale=0.6,  knoten/.style={
           draw=black, fill=black, thin, circle, inner sep=0.05cm}]
                     \draw[white] (0.5,2) -- (7,5);
           \node[knoten] (1) at (1,2) [label=below:$1$]{};
           \node[knoten] (2) at (2,2) [label=below:$2$]{};
           \node[knoten] (3) at (3,2) [label=below:$3$]{};
           \node[knoten] (4) at (4,2) [label=below:$4$]{};
           \node[knoten] (5) at (5,2) [label=below:$5$]{};
           \node[knoten] (6) at (6,2) [label=below:$6$]{};
               \node[] (7) at (3.5,1) [label=below:$(c)$]{};
           \draw[edge] (1) -- (2) -- (3) -- (4) -- (5) -- (6);
           \draw [color=black] (1)++(0.33,2.85) to ++(0,-2.6)  to ++(-0.7,0) to ++ (0,2.6);
           \draw [color=black] (2)++(0.33,2.85) to ++(0,-2.6)  to ++(-0.7,0) to ++ (0,2.6);
           \draw [color=black] (3)++(0.33,2.85) to ++(0,-2.6)  to ++(-0.7,0) to ++ (0,2.6);
           \draw [color=black] (4)++(0.33,2.85) to ++(0,-2.6)  to ++(-0.7,0) to ++ (0,2.6);
           \draw [color=black] (5)++(0.33,2.85) to ++(0,-2.6)  to ++(-0.7,0) to ++ (0,2.6);
           \draw [color=black] (6)++(0.33,2.85) to ++(0,-2.6)  to ++(-0.7,0) to ++ (0,2.6);
           \draw[dashed, line width=0.7pt, ->,  dash pattern=on 1pt off 1pt] (4.2,4.1) -- (4.8,3.7);
           \draw[line width=1pt, ->] (4,5.5) -- (4,4.4);
           \node[] (7) at (4,6.7) [label=below:ball $i+1$]{};
           \ball{1,2.6};
           \ball{2,2.6};
           \ball{2,3.1};
           \ball{3,2.6};
           \ball{3,3.1};
           \ball{4,2.6};
           \ball{4,3.1};
           \ball{4,3.6};
           \gball{4,4.1};
           \ball{5,2.6};
           \ball{5,3.1};
           \nball{5,3.6};
           \ball{6,2.6};
           \ball{6,3.1};
           \ball{6,3.6};
\end{tikzpicture}
\begin{tikzpicture}[auto, scale=0.6,  knoten/.style={
           draw=black, fill=black, thin, circle, inner sep=0.05cm}]
                     \draw[white] (0.5,2) -- (7,5);
           \node[knoten] (1) at (1,2) [label=below:$1$]{};
           \node[knoten] (2) at (2,2) [label=below:$2$]{};
           \node[knoten] (3) at (3,2) [label=below:$3$]{};
           \node[knoten] (4) at (4,2) [label=below:$4$]{};
           \node[knoten] (5) at (5,2) [label=below:$5$]{};
           \node[knoten] (6) at (6,2) [label=below:$6$]{};
           \node[] (7) at (3.5,1) [label=below:$(d)$]{};
           \draw[edge] (1) -- (2) -- (3) -- (4) -- (5) -- (6);
           \draw [color=black] (1)++(0.33,2.85) to ++(0,-2.6)  to ++(-0.7,0) to ++ (0,2.6);
           \draw [color=black] (2)++(0.33,2.85) to ++(0,-2.6)  to ++(-0.7,0) to ++ (0,2.6);
           \draw [color=black] (3)++(0.33,2.85) to ++(0,-2.6)  to ++(-0.7,0) to ++ (0,2.6);
           \draw [color=black] (4)++(0.33,2.85) to ++(0,-2.6)  to ++(-0.7,0) to ++ (0,2.6);
           \draw [color=black] (5)++(0.33,2.85) to ++(0,-2.6)  to ++(-0.7,0) to ++ (0,2.6);
           \draw [color=black] (6)++(0.33,2.85) to ++(0,-2.6)  to ++(-0.7,0) to ++ (0,2.6);
           \draw[line width=1pt, ->] (3,5.5) -- (3,3.9);
           \node[] (7) at (3,6.7) [label=below:ball $i+2$]{};
           \ball{1,2.6};
           \ball{2,2.6};
           \ball{2,3.1};
           \ball{3,2.6};
           \ball{3,3.1};
           \ball{4,2.6};
           \ball{4,3.1};
           \ball{4,3.6};
           \ball{5,2.6};
           \ball{5,3.1};
           \ball{5,3.6};
           \nball{3,3.6};
           \ball{6,2.6};
           \ball{6,3.1};
           \ball{6,3.6};
\end{tikzpicture}

\caption{Illustration of the local search allocation. Black circles represent the vertices 1--6 arranged as a path, and the yellow circles represent the balls of the process (the most recently allocated ball is marked red). Figure~$(a)$ shows the configuration after placing $i-1$ balls. As shown in Figure~$(b)$, ball $i$ born at vertex $4$ has two choices in the first step of the local search (vertices~$3$ or~$5$) and is finally allocated to vertex~$2$. Figure~$(c)$ and $(d)$ shows the placement of ball $i+1$ and $i+2$.
}

  \label{fig:localsearch}

\end{figure}
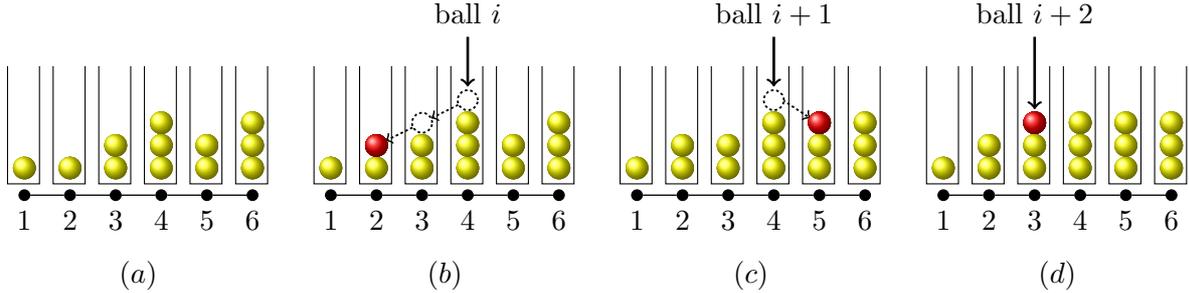

The main result in~\cite{BSSS13} establishes that when $G$ is an expander graph with bounded maximum degree, the maximum load after $n$ balls have been
allocated is $\Theta(\log\log n)$. Hence, local search allocation on bounded-degree expanders achieves the
same maximum load (up to constants) as in the $d$-choice process, but has the extra benefit of requiring only local information during the allocation. In~\cite{BSSS13}, it was also established that the maximum load is
$\Theta\left(\left(\frac{\log n}{\log\log n}\right)^\frac{1}{d+1}\right)$ on $d$-dimensional grids, and $\Theta(1)$ on regular graphs of degrees $\Omega(\log n)$.

\subsection{Results}
In this paper we derive upper and lower bounds for the maximum load, and propose the study of another natural quantity, which we refer to as the \emph{cover time}.
In order to state our results, we need to introduce the following two quantities that are related to the local neighborhood growth of $G$:
$$
   R_1 = R_1(G) = \min\{r\in \mathbb{N} \colon r |B_u^{r}|\log r \geq \log n \text{ for all $u\in V$}\}
$$
and
$$
   R_2 = R_2(G) = \min\{r\in \mathbb{N} \colon r |B_{u}^{r}| \geq \log n \text{ for all $u\in V$}\},
$$
where $B_u^{r}$ denotes the set of vertices within distance at most $r$ from vertex $u$.
Note that $R_1 \leq R_2$ for all $G$.
For the sake of clarity, we state our results here for \emph{vertex-transitive} graphs only.
In later sections we state our results in fullest generality, which will require a more refined definition of $R_1$ and $R_2$. We also highlight that for all the results below (and throughout this paper) we assume that ties are broken independently and uniformly at random; the impact of tie-breaking procedures in local search allocation was investigated in~\cite[Theorem~1.5]{BSSS13}.

\subsubsection*{Maximum load}
We derive an upper bound for the maximum load after $n$ balls have been allocated. Our bound holds for \emph{all} bounded-degree graphs, and is tight for vertex-transitive graphs (and, more generally, for graphs where the neighborhood growth is sufficiently homogeneous across different vertices).
\begin{thm}[Maximum load when $m=n$]\label{thm:maxload}
   Let $G$ be any vertex-transitive graph with bounded degrees. Then, with probability at least $1-n^{-1}$, the maximum load after $n$ balls have been allocated is $\Theta(R_1)$.
\end{thm}
Theorem~\ref{thm:maxload} is a special case of Theorem~\ref{thm:maxload2},
which gives a more precise version of the result above and generalizes it to non-transitive graphs;
in particular, we obtain that for any graph with bounded degrees the maximum load is $\Oh(R_1)$ with high probability.
We state and prove Theorem~\ref{thm:maxload2}
in Section~\ref{sec:maxload}.

Note that for bounded-degree expanders we have $R_1=\Theta(\log\log n)$, and for $d$-dimensional grids we have $R_1=\Theta\left(\left(\frac{\log n}{\log\log n}\right)^\frac{1}{d+1}\right)$. Hence the results for bounded-degree graphs in~\cite{BSSS13} are special cases of Theorems~\ref{thm:maxload} and~\ref{thm:maxload2}. Furthermore, the proof of Theorems~\ref{thm:maxload} and~\ref{thm:maxload2} uses different techniques (it follows by a subtle coupling with the $1$-choice process) and is substantially shorter than the proofs in~\cite{BSSS13}.

Our second result establishes an upper bound for the maximum load when $m\geq n$. We point out that all other results known so far were limited to the case $m=n$. We establish that, when $m=\Omega(R_2 n)$, then the maximum load is of order $\Theta(m/n)$ (i.e., the same order as the average load).
We note that the difference between the maximum load and the average load for the local search allocation is always bounded above by the diameter of the graph (see Lemma~\ref{lem:balanced} below). This is in some sense similar to the $d$-choice process, where the difference between the maximum load and the average load does not depend on $m$~\cite{BCSV06}.
\begin{thm}[Maximum load when $m\geq n$]\label{thm:highdensity}
   Let $G$ be any graph with bounded degrees. Then for any $m\geq n$, with probability at least $1-n^{-1}$, the maximum load after $m$ balls have been allocated is $\Oh(\frac{m}{n}+R_2)$.
\end{thm}

\subsubsection*{Cover time}
We propose to study the following natural quantity related to any process based on allocating balls into bins.
Define the \emph{cover time} as the first time at which all bins have at least one ball allocated to them.
This is in analogy with cover time of random walks on graphs, which is the first time at which the random walk has visited all vertices of the graph.
Note that for the $1$-choice process, the cover time corresponds to the time of a \emph{coupon collector} problem,
which is known to be $n \log n +\Theta(n)$~\cite[Section~2.4.1]{MU05}.
For the $d$-choice process with $d=\Theta(1)$, we obtain that the cover time is also of order $n\log n$.
We show that for the local search allocation the cover time can be much smaller than $n \log n$.

Our next theorem establishes that the cover time for vertex-transitive bounded-degree graphs is $\Theta(R_2 n)$ with high probability.
\begin{thm}[Cover time for bounded-degree graphs]\label{thm:covertime}
   Let $G$ be any vertex-transitive graph with bounded degrees. Then, with probability at least $1-n^{-1}$, the cover time of local search allocation on $G$ is $\Theta(R_2n)$.
\end{thm}
The theorem above is a special case of Theorem~\ref{thm:covertime2},
which we state and prove in Section~\ref{sec:covertime}; in particular, we establish there that the upper bound of $\Oh(R_2n)$ holds for all bounded-degree graphs.
 Since $R_2=\Oh(\sqrt{\log n})$ for all connected graphs, it follows that the cover time for any connected, bounded-degree graph is at most $\Oh(n \sqrt{\log n})$, which is significantly smaller than the cover time of the $d$-choice process for any $d=\Theta(1)$. In particular, we obtain $R_2=\Theta(\log\log n)$ for bounded-degree expanders, and $R_2=\Theta\Big((\log n)^\frac{1}{d+1}\Big)$ for $d$-dimensional grids.

Our final result provides a general upper bound on the cover time for dense graphs.
Theorem~\ref{thm:dense} below is a special case of Theorem~\ref{thm:densefull}, which gives an upper bound on the cover time for all regular graphs.
We state and prove Theorem~\ref{thm:densefull} in Section~\ref{sec:covertime}.
\begin{thm}[Cover time for dense graphs]\label{thm:dense}
   Let $G$ be any $d$-regular graph with $d=\Omega(\log n \log\log n)$.
   Then, with probability at least $1-n^{-1}$, the cover time is $\Theta(n)$.
\end{thm}

\section{Background and notation}\label{sec:prelim}
In this section we recall some basic properties of the local search allocation that will be useful in our proofs.

Let $G=(V,E)$ be an undirected, not necessarily connected, graph with $n$ vertices, and let $\Delta$ be the maximum degree of $G$. 
We assume that, in the local search allocation, ties are broken independently and uniformly at random.

We denote by $d_{G}(u,v)$ the distance between $u$ and $v$ in $G$ and define
$d_{G}(v,S)= \min_{v' \in S} d_{G}(v,v')$ for any non-empty subset $S \subseteq V$. Further, we define $B_{u}^{r}=\{v \in V \colon d_{G}(v,u) \leq r\}$ and for any non-empty set $S \subseteq V$, $B_{S}^{r}=\{v \in V \colon d_{G}(v,S) \leq r\}.$

For each $m\geq 0$ and vertex $v\in V$, let $X_v^{(m)}$ denote the load of $v$ (i.e., the number of balls allocated to $v$) after $m$ balls have been allocated.
Initially we have $X_v^{(0)}=0$ for all $v\in V$ and, for any $m\geq0$, we have $\sum_{v\in V}X_v^{(m)}=m$.
Denote by $X_{\max}^{(m)}$ the maximum load after $m$ balls have been allocated; i.e.,
$$
   X_{\max}^{(m)} = \max_{v\in V} X_v^{(m)}.
$$
Also, denote by $T_\cov=T_\cov(G)$ the \emph{cover time} of $G$, which we define as the first time at which all vertices have load at least $1$.
More formally,
$$
   T_\cov = \min\{m\geq 0 \colon X_v^{(m)} \geq 1 \text{ for all $v\in V$}\}.
$$

Let $U_i\in V$ denote the birthplace of ball $i$, and for each $m\geq 0$ and $v\in V$, let $\bar X_v^{(m)}$ denote the load of $v$ after $m$ balls have been allocated according to the $1$-choice process. Let $\bar X_{\max}^{(m)}$ denote the maximum load for the 1-choice process. More formally,
\begin{equation}
   \bar X_v^{(m)}=\sum_{i=1}^m \ind{U_i=v}
   \quad\text{and}\quad
   \overline{X}_{\max}^{(m)}=\max_{v\in V}\overline{X}_v^{(m)}
   \label{eq:1choice}
\end{equation}

For  vectors $A=(a_1,a_2,\ldots,a_n)$ and $A'=(a_1',a_2',\ldots,a_n')$ such that $\sum_{i=1}^n a_i = \sum_{i=1}^n a_i'$,
we say that $A$ \emph{majorizes} $A'$ if, for each $\kappa=1,2,\ldots,n$, the sum of the $\kappa$ largest entries
of $A$ is at least the sum of the $\kappa$ largest  entries of $A'$. More formally,
if $j_1,j_2,\ldots,j_n$ are distinct numbers such that
$a_{j_1}\geq a_{j_2} \geq \cdots \geq a_{j_n}$ and
$j_1',j_2',\ldots,j_n'$ are distinct numbers  such that
$a'_{j_1'}\geq a'_{j_2'} \geq \cdots \geq a'_{j_n'}$, then $A$ majorizes $A'$ if
\begin{equation}
   \sum_{i=1}^\kappa a_{j_i} \geq \sum_{i=1}^\kappa a'_{j_i'}
   \quad\text{for all $\kappa=1,2,\ldots,n$.}
   \label{eq:majorization}
\end{equation}

The lemma below establishes that the load vector obtained by the $1$-choice process majorizes the load vector obtained by the local search allocation.
As a consequence, we have that $X_{\max}^{(n)} = \Oh\left(\frac{\log n}{\log\log n}\right)$ and $T_\cov = \Oh\left(n\log n\right)$ for all $G$.
Later, in Section~\ref{sec:key}, we state and prove Lemma~\ref{lem:majorization}, which is a generalization of Lemma~\ref{lem:1choice}.
\begin{lem}[{Comparison with $1$-choice}]\label{lem:1choice}
   For any fixed $k\geq 0$,
   we can couple $X^{(k)}$ and $\bar X^{(k)}$ so that, with probability 1, $\bar X^{(k)}$ majorizes $X^{(k)}$.
   Consequently, we have that, for all $k\geq 0$, $\bar X_{\max}^{(k)}$ stochastically dominates $X_{\max}^{(k)}$.
\end{lem}

For any $v\in V$, let $N_v$ be the set of neighbors of $v$ in $G$.
The next lemma establishes that the local search allocation always maintains a \emph{smoothed} load vector in the sense that the load of
any two adjacent vertices differs by at most $1$.
\begin{lem}[{Smoothness}]\label{lem:balanced}
   For any $k\geq 0$, any $v\in V$ and any $u\in N_v$, we have that $|X_v^{(k)}-X_{u}^{(k)}|\leq 1.$
\end{lem}
\begin{proof}
   In order to obtain a contradiction, suppose that $X_v^{(k)} \geq X_u^{(k)} +2$, and let $j$ be the last ball allocated to
   $v$. Then, we have that
   \[
     X_v^{(j-1)} = X_v^{(k)} - 1 \geq X_u^{(k)} + 1 \geq X_u^{(j-1)} + 1.
   \]
   Therefore, the moment the $j$th ball is born, vertex $v$ has at least one neighbor with load
   strictly smaller than $v$. Therefore, ball $j$ is not allocated to $v$, establishing a contradiction.
\end{proof}

The next lemmas establish that the load vector $X^{(n)}$ satisfies a Lipschitz and monotonicity condition.
\begin{lem}[{Lipschitz property}]\label{lem:lipschitz}
   Let $k\geq 1$ be fixed and $u_1,u_2,\ldots,u_k \in V$ be arbitrary.
   Let $(X^{(k)}_v)_{v\in V}$ be the load of the vertices of $G$ after the local search allocation places $k$ balls with birthplaces $u_1,u_2,\ldots,u_k$.
   Let $i\in \{1,2,\ldots,k\}$ be fixed, and let $(Y^{(k)}_v)_{v\in V}$ be the load of the vertices of $G$ after the local search allocation places
   $k$ balls
   with birthplaces $u_1,u_2,\ldots,u_{i-1},u_i',u_{i+1},u_{i+2},\ldots,u_k$, where $u_i'\in V$ is arbitrary.
   In other words, $(Y_v^{(k)})_{v\in V}$ is obtained from $(X_v^{(k)})_{v\in V}$ by changing the
   birthplace of the $i$th ball from $u_i$ to $u_i'$.
   Then, there exists a coupling such that, with probability $1$,
   \begin{equation}
      \sum_{v\in V} \left|X_v^{(k)} - Y_v^{(k)}\right| \leq 2.
      \label{eq:lipschitz}
   \end{equation}
\end{lem}
\begin{proof}
   We refer to the process defining the variables $X^{(k)}$ as the $X$ process, and the
   process defining the variables $Y^{(k)}$ as the $Y$ process.
   For each $v\in V$ and $i\geq 1$, we define $\xi_v^{(i)}$ to be an independent and uniformly random permutation of
   the neighbors of $v$. We use this permutation for both the $X$ and $Y$ processes to break ties when ball $i$ is at vertex $v$.
   Then, since the first $i-1$ balls have the same birthplaces in both processes, we have that
   \begin{equation}
      X_v^{(i-1)} = Y_v^{(i-1)} \quad\text{ for all $v\in V$}.
      \label{eq:prefixunchanged}
   \end{equation}
   Now, when adding the $i$th ball, we let $v_i$ be the vertex to which this ball is allocated in the $X$ process and $v_i'$ be the vertex to which this
   ball is allocated in the $Y$ process.
   If $v_i=v_i'$, then $X_u^{(i)}=Y_u^{(i)}$ for all $u\in V$ and~\eqref{eq:lipschitz} holds.
   More generally, we have that
   \begin{equation} \label{eq:desiredproperty}
   \left\{ \begin{aligned}
            X_{v_i}^{(i)} &=Y_{v_i}^{(i)}+\ind{v_i \neq v_{i}'} \\
            Y_{v'_i}^{(i)} &=X_{v'_i}^{(i)}+\ind{v_i \neq v_{i}'}\\
            X_{u}^{(i)}&=Y_{u}^{(i)} \text{ for $u \in V\setminus \{v_i,v_i'\}$.}
                             \end{aligned} \right.
                             \end{equation}
   If $i=k$, then this implies~\eqref{eq:lipschitz} and the lemma holds.

   For the case $i<k$, we add ball $i+1$ and are going to define $v_{i+1}$ and $v_{i+1}'$
   so that~\eqref{eq:desiredproperty} holds with $i$ replaced by $i+1$. Then the proof of the lemma is completed by induction.
   We assume that $v_i\neq v_{i}'$, otherwise~\eqref{eq:lipschitz} clearly holds.
   We note that $v_{i+1}$ and $v_{i+1}'$ will not \emph{necessarily} be in the same way as $v_i$ and $v_i'$. The role of $v_{i+1}$ and $v_{i+1}'$ is to be the only vertices whose loads in the $X$ and $Y$ processes are different.
   The definition of $v_{i+1}$ and $v_{i+1}'$ will vary depending on the situation.
   For this, let ball $i+1$ be born at $u_{i+1}$ and define $w$
   to be the vertex on which ball $i+1$ is allocated in the $X$ process and
   $w'$ to be the vertex on which ball $i+1$ is allocated in the $Y$ process.
   We can assume that $w\neq w'$, otherwise~\eqref{eq:desiredproperty} holds with $i$ replaced by $i+1$ by setting $v_{i+1}=v_i$ and $v_{i+1}'=v_i'$.

   Now we analyze ball $i+1$. It is crucial to note that, during the local search of ball $i+1$, if it does not enter $v_i$ in the $Y$ process
   and does not enter $v_i'$ in the $X$ process, then ball $i+1$ follows the same path in both processes. Since we are in the case $w\neq w'$,
   we can assume without loss of generality that ball $i+1$ eventually visits $v_i$ in the $Y$ process. In this case, since the local search performed by ball $i$
   in the $X$ process stops at vertex $v_i$, we have that $v_i$ is a local minimum for ball $i+1$ in the $Y$ process, which implies
   that $w'=v_i$.
   (The case when ball $i+1$ visits $v_i'$ in the $X$ process follows by a symmetric argument.)
   So, since $w\neq w'$, we have $X_{v_i}^{(i+1)}=Y_{v_i}^{(i+1)}$.
   Then we let $v_{i+1}=w$. If $w=v_i'$, we set $v_{i+1}'=w$ and~\eqref{eq:desiredproperty} holds since $X_u^{(i+1)}=Y_u^{(i+1)}$ for all $u\in V$.
   Otherwise we set $v_{i+1}'=v_i'$,
   and~\eqref{eq:desiredproperty} holds as well.
\end{proof}

\begin{lem}[{Monotonicity}]\label{lem:monotonicity}
   Let $k\geq 1$ be fixed and $u_1,u_2,\ldots,u_k \in V$ be arbitrary.
   Let $(X^{(k)}_v)_{v\in V}$ be the load of the vertices after $k$ balls are allocated with birthplaces $u_1,u_2,\ldots,u_k$.
   Let $i\in \{1,2,\ldots,k\}$ be fixed, and let $(Z^{(i,k)}_v)_{v\in V}$ be the load of the vertices of $G$ after $k-1$ balls are allocated
   with birthplaces $u_1,u_2,\ldots,u_{i-1},u_{i+1},u_{i+2},\ldots,u_k$.
   In other words, $Z_v^{(i,k)}$ is obtained from $X_v^{(k)}$ by removing ball $i$.
   There exists a coupling such that, with probability $1$,
   $$
      \sum_{v\in V} \left|X_v^{(k)} - Z_v^{(i,k)}\right| =1.
   $$
\end{lem}
\begin{proof}
   Let $G'$ be the graph obtained from $G$ by adding an isolated node $w$; i.e., $G'$ has vertex set $V\cup \{w\}$ and the same edge set as $G$.
   Applying \lemref{lipschitz} to $G'$ with the same choice of $u_1,\ldots,u_k\in V$ and with $u_{i}'=w$ gives
   \[
      \sum_{v \in V \cup \{w\} } \left| X_v^{(k)} - Y_v^{(k)} \right| = 2.
   \]
   Since $Y_{w}^{(k)} = 1$, $X_w^{(k)} = 0$ and $Z_{v}^{(i,k)} = Y_v^{(k)}$ for any $v \in V$, we conclude that
   \[
       \sum_{v \in V } \left| X_v^{(k)} - Z_v^{(i,k)} \right| = \sum_{v \in V } \left| X_v^{(k)} - Y_v^{(k)} \right| = 1. \qedhere
   \]
\end{proof}

In many of our proofs we analyze a continuous-time variant where the number of balls is \emph{not} fixed, but is given by a Poisson random variable with mean $m$. Equivalently, in this variant balls are born at each vertex according to a Poisson process of rate $1/n$.
We refer to this as the \emph{Poissonized} version. We will use the Poissonized versions of both the local search allocation and the $1$-choice process in our proofs. Since the probability that a mean-$m$ Poisson random variable takes the value $m$ is of order $\Theta(m^{-1/2})$ we obtain the
following relation.
\begin{lem}\label{lem:poisson}
   Let $\mathcal{A}$ be an event that holds for the Poissonized version of the local search allocation (respectively, $1$-choice process) with probability $1-\epsilon$ for some $\epsilon\in(0,1)$. Then, the probability that
   $\mathcal{A}$ holds for the non-Poissonized version of the local search allocation (respectively, $1$-choice process) is at least $1-\Oh(\epsilon \sqrt{m})$.
\end{lem}


\section{Key technical argument}\label{sec:key}

In this section we prove a key technical result (Lemma~\ref{lem:majorization} below) that will play a central role in our proofs later.

Let $\mu\colon V \to \mathbb{Z}$ be any integer function 
on the vertices of $G$ that satisfies the following property:
\begin{equation}
   \text{for any two neighbors $u,v \in V$, we have $|\mu(u)-\mu(v)| \leq 1$}.
   \label{eq:mu}
\end{equation}
We see $\mu$ as an initial attribution of weights to the vertices of $G$. Then, for any $m\geq 1$, after $m$ balls are allocated,
we define the weight of vertex $v$ by
\begin{equation}
   W_v^{(m)} = X_{v}^{(m)} + \mu(v).
   \label{eq:phase}
\end{equation}
Note that for any $m\geq 1$ and $v\in V$, we have that $W_v$ can increase by at most one after each step;
i.e., $W_v^{(m)}\in\{W_v^{(m-1)},W_v^{(m-1)}+1\}$. The lemma below establishes that a ball cannot be allocated to
a vertex with larger weight than the vertex where the ball is born.

\begin{lem}\label{lem:phasesearch}
   Let $m\geq 1$ and denote by $v$ the vertex where ball $m$ is born (i.e., $v=U_m$).
   Let $v'$ be the vertex where ball $m$ is allocated. Then,
   $
      W_{v'}^{(m-1)}\leq W_{v}^{(m-1)}.
   $
\end{lem}
\begin{proof}
   Assume that $v\neq v'$, thus the local search of ball $m$ visits at least two vertices.
   Let $w$ be the second vertex visited during the local search. Since $v$ and $w$ are neighbors in $G$, we have
   $$
      W_w^{(m-1)}
      = X_w^{(m-1)} + \mu(w)
      = X_v^{(m-1)} - 1 + \mu(w)
     \leq X_v^{(m-1)} + \mu(v)
     = W_v^{(m-1)}.
   $$
   Proceeding inductively for each step of the local search we obtain $W_{v'}^{(m-1)} \leq W_v^{(m-1)}$.
\end{proof}

Now we use the definition of majorization from~\eqref{eq:majorization}.
Let $\bar W_v^{(m)}$ be the weight of vertex $v$ after $m$ balls are allocated according to the $1$-choice process;
i.e., $\bar W_v^{(m)} = \bar X_v^{(m)} + \mu(v)$ for all $v\in V$.
The lemma below extends the result of Lemma~\ref{lem:1choice} to the weights of the vertices; 
Lemma~\ref{lem:1choice} can be obtained from Lemma~\ref{lem:majorization} by setting $\mu(v)=0$ for all $v\in V$.
\begin{lem}\label{lem:majorization}
   For any fixed $m\geq 0$,
   we can couple $W^{(m)}$ and $\bar W^{(m)}$ so that, with probability 1, $\bar W^{(m)}$ majorizes $W^{(m)}$.
\end{lem}
For the proof of this lemma, we need the following result from \cite{ABKU99}.
\begin{lem}[{\cite[Lemma~3.4]{ABKU99}}]\label{lem:abku}
Let $v=(v_1,v_2,\ldots,v_n)$, $u=(u_1,u_2,\ldots,u_n)$ be two vectors such that $v_1 \geq v_2 \geq \cdots \geq v_n$ and $u_1 \geq u_2 \geq \cdots \geq u_n$. If $v$ majorizes $u$, then also $v+e_{i}$ majorizes $u+e_{i}$, where $e_{i}$ is the $i$th unit vector.
\end{lem}

\begin{proof}[Proof of Lemma~\ref{lem:majorization}]
The proof is by induction on $m$. Clearly, for $m=0$, we have $W_v^{(0)} = \bar W_v^{(0)}=\mu(v)$ for all $v\in V$.
Now, assume that we can couple $W^{(m-1)}$ with $\bar W^{(m-1)}$ so that $\bar W^{(m-1)}$ majorizes $W^{(m-1)}$.
Let $i_1,i_2,\ldots, i_n$ be distinct elements of $V$ so that
\[
W_{i_1}^{(m-1)} \geq W_{i_2}^{(m-1)} \geq \cdots \geq W_{i_n}^{(m-1)}.
\]
Similarly, let $j_1,j_2,\ldots, j_n$ be distinct elements of $V$ so that
\[
\bar W_{j_1}^{(m-1)} \geq \bar W_{j_2}^{(m-1)} \geq \cdots \geq \bar W_{j_n}^{(m-1)}.
\]
Let $\ell$ be a uniformly random integer from $1$ to $n$.
Then, for the process $(W_v^{(m)})_{v\in V}$, let the birthplace of ball $m$ be vertex $i_\ell$ and for the process $(\bar W_v^{(m)})_{v\in V}$, let the birthplace of ball $m$ be $j_\ell$. For the process $(W_v^{(m)})_{v\in V}$, ball $m$ may not necessarily be allocated at vertex $i_{\ell}$, so let us define $\iota$ as the integer so that
$i_{\iota}$ is the vertex to which ball $m$ is allocated.

In order to prove that $\bar W^{(m)}$ majorizes $W^{(m)}$,
let us define by $\tilde W^{(m)}$ the load vector which is obtained
from $W^{(m-1)}$ by allocating ball $m$ to vertex $i_{\ell}$ (the birthplace of ball $m$). Applying \lemref{abku} gives that $\bar{W}^{(m)}$ majorizes $\tilde{W}^{(m)}$, since by the induction hypothesis $\bar{W}^{(m-1)}$ majorizes $W^{(m-1)}$. Next observe that
\[
  W^{(m)} = \tilde{W}^{(m)} - e_{i_{\ell}} + e_{i_{\iota}},
\]
so we obtain the vector $W^{(m)}$ from $\tilde{W}^{(m)}$ by removing one from vertex $i_{\ell}$ and adding one to vertex $i_{\iota}$.
By \lemref{phasesearch}, we have $W_{i_{\iota}}^{(m-1)} \leq W_{i_{\ell}}^{(m-1)}$. This implies $\tilde{W}_{i_{\ell}}^{(m)}
= W_{i_{\ell}}^{(m-1)} + 1
\geq W_{i_{\iota}}^{(m-1)} + 1$ and in turn that $\tilde{W}^{(m)}$  majorizes $W^{(m)}$. Combining this with the insight that $\bar{W}^{(m)}$ majorizes $\tilde{W}^{(m)}$ implies that $\bar{W}^{(m)}$ majorizes $W^{(m)}$. This completes the induction and the proof.
\end{proof}

Now we illustrate the usefulness of the above result by relating the probability of a vertex to have a certain load
with the probability that balls are born in a neighborhood around a vertex.
Recall that the load vector is smooth (cf.\ Lemma~\ref{lem:balanced}), which means that if a vertex $v$ has load $\ell$,
then a vertex at distance $r$ from $v$ has load at least $\ell-r$ and at most $\ell+r$.

\begin{lem}\label{lem:pyramid}
   For any $v\in V$, and any $\ell,m\geq 1$, we have
   $$
      \Pro{X_v^{(m)} \geq \ell} \geq \Pro{\bigcap\nolimits_{w\in B_v^{\ell-1}} \Big\{ \bar X_{w}^{(m)} \geq \ell-d_G(v,w)\Big\}}
   $$
   and
   $$
      \Pro{X_v^{(m)} \geq \ell} \leq \Pro{\bigcup\nolimits_{w\in V} \Big\{ \bar X_{w}^{(m)} \geq \ell+d_G(v,w)\Big\}}.
   $$
\end{lem}
\begin{proof}
   For the first inequality, set $\mu(w)=d_G(v,w)$ for all $w\in V$.
   Let $\mathcal{A}^{(m)}$ be the event that all vertices have weight at least $\ell$ after $m$ balls are allocated, and
   let $\bar{\mathcal{A}}^{(m)}$ be the same event for the $1$-choice process.
   In symbols $\mathcal{A}^{(m)}=\{\min_{u\in V}W_u^{(m)}\geq \ell\}$
   and $\bar{\mathcal{A}}^{(m)}=\{\min_{u\in V}\bar W_u^{(m)}\geq \ell\}$.
   By Lemma~\ref{lem:majorization}, we have that $\Pro{\mathcal{A}^{(m)}} \geq \Pro{\bar {\mathcal{A}}^{(m)}}$.
   Clearly, we have $\mathcal{A}^{(m)} \subseteq \{X_v^{(m)} \geq \ell\}$, but the two events are in fact equal since, by the smoothness of the load
   vector (cf.\ Lemma~\ref{lem:balanced}), $\{X_v^{(m)} \geq \ell\}$ implies $\mathcal{A}^{(m)}$.
   The proof is then complete since
   $\bar{\mathcal{A}}^{(m)}=\bigcap\nolimits_{w\in B_v^{\ell-1}} \Big\{ \bar X_{w}^{(m)} \geq \ell-d_G(v,w)\Big\}$.

   For the second inequality, set $\mu(w)=-d_G(v,w)$ for all $w\in V$.
   Then define $\mathcal{B}^{(m)}$ to be the event that there exists at least one vertex with weight at least $\ell$ after $m$ balls are allocated, and
   let $\bar{\mathcal{B}}^{(m)}$ be the corresponding event for the $1$-choice process.
   Thus, $\mathcal{B}^{(m)}=\{\max_{u\in V}W_u^{(m)}\geq \ell\}$
   and $\bar{\mathcal{B}}^{(m)}=\{\max_{u\in V}\bar W_u^{(m)}\geq \ell\}$.
   Similarly as for the event $\mathcal{A}^{(m)}$, we have that the events $\{X_v^{(m)} \geq \ell\}$ and $\mathcal{B}^{(m)}$ are identical.
   Applying Lemma~\ref{lem:majorization} we obtain that
   $\Pro{\mathcal{B}^{(m)}} \leq \Pro{\bar {\mathcal{B}}^{(m)}}=\Pro{\bigcup\nolimits_{w\in V} \Big\{ \bar X_{w}^{(m)} \geq \ell+d_G(v,w)\Big\}}$.
\end{proof}
\begin{rem}{\rm
   The lemma above states that one can \emph{couple} $\{X_v^{(m)}\}_{v\in V}$ and $\{\bar X_v^{(m)}\}_{v\in V}$ so that if
   $\bar X_{w}^{(m)} \geq \ell-d_G(v,w)$ for all $w\in B_v^{\ell-1}$, then $X_v^{(m)} \geq \ell$.
   However, this is not necessarily achieved with the ``trivial''
   coupling where each ball is born at the same vertex for both processes $\{X_v^{(m)}\}_{v\in V}$ and $\{\bar X_v^{(m)}\}_{v\in V}$.
   In other words, knowing that the number of balls born at vertex $w$ is at least $\ell-d_G(v,w)$ for all $w\in B_v^\ell$ does \emph{not}
   imply that $X_v^{(m)} \geq \ell$.
}\end{rem}

Now we extend the proof of Lemma~\ref{lem:pyramid} to derive an upper bound on the load of a subset of vertices.
\begin{pro}\label{pro:upperbound}
   Let $S\subset V$ be fixed and $\Delta$ be the maximum degree in $G$.
   Then, for all $m\geq n$ and $\ell \geq \frac{300\Delta m}{n}$ we have
   $$
      \Pro{\sum_{v\in S} X_v^{(m)} \geq \ell|S|} \leq 4\exp\left(-\frac{|S|\ell}{14} \log\left(\frac{\ell n}{m}\right) \right)+\exp\left(-\frac{m}{4}\right).
   $$
   The above inequality implies that, for any given $u\in V$,
   $$
      \Pro{X_u^{(m)} \geq 2\ell} \leq 4\exp\left(-\frac{|B_u^\ell|\ell}{14} \log\left(\frac{\ell n}{m}\right) \right)+\exp\left(-\frac{m}{4}\right).
   $$
\end{pro}
\begin{proof}
   For any $v\in V$, define
   $\mu(v) = -d_G(v,S)$ and (cf.~\eqref{eq:phase}) $W_v^{(m)} = X_v^{(m)}+\mu(v)$.
   Let $K_S^{(m)}$ be the sum of the weights of the $|S|$ vertices with largest weights after $m$ balls are allocated, and $\bar K_S^{(m)}$ be the
   corresponding value for the 1-choice process.
   Then,
   $$
      \sum_{v\in S} X_v^{(m)}
      = \sum_{v\in S} W_v^{(m)}
      \leq K_S^{(m)}
      \leq \bar K_S^{(m)},
   $$
   where the last step follows by majorization (cf.\ Lemma~\ref{lem:majorization}).
   Let $\hat W_v^{(k)}$ be the weight of vertex $v$ for the Poissonized version of the $1$-choice process with expected number of balls equal to $k$, and
   $\hat K_S^{(k)}$ be the sum of the weights of the $|S|$ vertices with largest weight for this Poissonized version.
   If the Poissonized version with $k=2m$ allocates at least $m$ balls, then we can couple the allocations of the first $m$ balls with the allocation in the non-Poissonized version of the $1$-choice process, and it holds that $\hat K_S^{(2m)} \geq \bar K_S^{(m)}$. Hence
   by the first statement of Lemma~\ref{lem:AlonSpencer} we have that
   \begin{equation}
      \Pro{\hat K_S^{(2m)} \geq \bar K_S^{(m)}}
      \geq 1-\exp\left(-\frac{m}{4}\right).
      \label{eq:coupling}
   \end{equation}
   From now on, we consider only the Poissonized version.
   Let $\tilde K^{(2m)}$ be the sum of the weights of the vertices with weight at least $\ell/16$. More formally,
   $\tilde K^{(2m)} = \sum_{v \in V \colon \hat W_v^{(2m)}\geq \ell/16} \hat W_v^{(2m)}$.
   Then, we have that, on the event $\hat K_S^{(2m)} \geq \bar K_S^{(m)}$,
   $$
      \sum_{v\in S} X_v^{(m)}
      \leq \hat K_S^{(2m)}
      \leq \frac{\ell}{16}|S| + \tilde K^{(2m)}.
   $$

   We can estimate the weight of vertices that reach weight $\ell/16$ as follows.
   For each vertex, let balls arrive according to a rate-1 Poisson point process up to time $2m/n$ or until the vertex reaches weight $\ell/16$, whatever happens first.
   Then, if the vertex reaches weight $\ell/16$,
   continue adding balls for an additional time interval of length $2m/n$. This construction stochastically dominates the
   weight of the vertices by the memoryless property of Poisson random variables. The probability that a vertex $v$ with $\mu(v)=-k$
   reaches weight $\ell/16$ is
   $$
      \sum_{x=\ell/16+k}^\infty \frac{\ce^{-2m/n}(2m/n)^x}{x!}
      \leq \sum_{x=\ell/16+k}^\infty \left(\frac{2m \ce }{n x}\right)^{x}
      \leq 2 \left(\frac{2m \ce}{n(\ell/16+k)}\right)^{\ell/16+k},
   $$
   since $\frac{2m \ce}{n(\ell/16+k)} \leq \frac{1}{2}$ for all $k\geq 0$ and $x! \geq (x/\ce)^x$ for any integer $x$.
   Now any Bernoulli random variable with mean $p\leq 1/2$ is stochastically dominated by a Poisson random variable with mean $2p$, which follows from the fact that $\ce^{-2p} \leq 1-p$ for $0 \leq p \leq 1/2$.
   Using this, and denoting by $N_S^k$ the set of vertices at distance $k$ from $S$, we have that the
   number of vertices reaching weight $\ell/16$ is a Poisson random variable of mean
   \begin{align*}
      \sum_{k\geq 0}|N_S^{k}|4 \left(\frac{2m \ce}{n(\ell/16+k)}\right)^{\ell/16+k}
      &\leq 4 |S| \left(\frac{32 m \ce}{n\ell}\right)^{\ell/16} \sum_{k\geq 0} \Delta^k\left(\frac{2m \ce}{n(\ell/16+k)}\right)^{k}\\
      &\leq 8 |S| \left(\frac{32 m \ce}{n \ell}\right)^{\ell/16},
   \end{align*}
   for large enough $\ell\geq 300\Delta m/n$.
   Then the probability that the number of vertices reaching weight $\ell/16$ is
   larger than $8|S|$ is at most
   \begin{align*}
      \sum_{k\geq 8|S|} \left(\frac{8\ce |S| \left(\frac{32 m \ce}{n \ell}\right)^{\ell/16}}{k}\right)^k
      &\leq 2\left(\frac{8\ce |S| \left(\frac{32 m \ce}{n \ell}\right)^{\ell/16}}{8|S|}\right)^{8|S|}\\
      &= 2\exp\left(-8|S| \left(\frac{\ell}{16}\log\left(\frac{\ell n}{32 m \ce}\right)-1\right)\right),
   \end{align*}
   since $\frac{8\ce |S| \left(\frac{32 m \ce}{n\ell}\right)^{\ell/16}}{8|S|}=\ce \left(\frac{32m \ce}{n\ell}\right)^{\ell/16}\leq\frac{1}{2}$.
   Using that $\ell\geq 300\Delta m/n$, we have
   $$
      \frac{\ell}{16}\log\left(\frac{\ell n}{32 m e}\right)-1
      \geq \frac{\ell}{32}\log\left(\frac{\ell n}{m}\right)-1
      \geq \frac{\ell}{96}\log\left(\frac{\ell n}{m}\right).
   $$
   Putting the last two equations together, we obtain that
   \begin{equation}
      \Pro{\text{more than $8|S|$ vertices reach weight $\frac{\ell}{16}$}}
      \leq 2\exp\left(-\frac{\ell |S|}{12}\log\left(\frac{\ell n}{m}\right)\right).
      \label{eq:part1}
   \end{equation}
   If the event above occurs, then $\tilde K_S^{(2m)}$ is stochastically dominated by
   $8|S| \cdot \frac{\ell}{16}=\frac{\ell|S|}{2}$ plus a Poisson random variable of mean $8|S| \cdot \frac{2m}{n} = 16 |S| \frac{m}{n}$, which is larger than $\frac{7\ell|S|}{16}$ with probability at most
   \begin{align}
      \sum_{k=\frac{7\ell|S|}{16}}^\infty \left(\frac{16\ce|S|m}{nk}\right)^k
      &\leq 2 \left(\frac{16\cdot 16 \ce|S|m}{7n\ell|S|}\right)^{\frac{7\ell|S|}{16}}\nonumber\\
      &=2\exp\left(-\frac{7\ell|S|}{16} \log\left(\frac{7 \ell n}{256\ce m}\right) \right)
      \leq 2\exp\left(-\frac{7\ell|S|}{96} \log\left(\frac{\ell n}{m}\right) \right).
      \label{eq:part2}
   \end{align}
   Therefore, by summing the right-hand sides of~\eqref{eq:part1} and~\eqref{eq:part2},
   with probability at least $1-4\exp\left(-\frac{\ell|S|}{14} \log\left(\frac{\ell n}{m}\right) \right)$, we have
   $\tilde K_S^{(2m)} \leq \frac{\ell|S|}{2}+\frac{7\ell|S|}{16}\leq\frac{15\ell|S|}{16}$.
   This and the fact that $\hat K_S^{(2m)} \leq \frac{\ell|S|}{16}+\tilde K_S^{(2m)}$, together with~\eqref{eq:coupling},
   establish the first part of the lemma.

   The second part of Proposition~\ref{pro:upperbound} holds by setting $S=B_u^\ell$. If $u$ has load $k>\ell$, then the total number of
   balls allocated to $B_u^\ell$ is at least
   $$
      \sum_{i=0}^\ell (k-i)|N_u^i|
      \geq (k-\ell)|B_u^\ell|.
   $$
   Then setting $k=2\ell$ and applying the first part of the proposition yields the result.
\end{proof}

\section{Maximum Load}\label{sec:maxload}
We start stating a stronger version of Theorem~\ref{thm:maxload} which also holds for non-transitive graphs.
For $\gamma\in(0,1/2]$, let
\begin{align*}
   R_1^{(\gamma)} = R_1^{(\gamma)}(G) = \max\big\{r\in\mathbb{N} \colon & \text{there exists } S\subseteq V \text{ with } |S|\geq n^{\frac{1}{2}+\gamma} \\
   & \text{ such that } r |B_u^{r}|\log r < \log n \text{ for all $u\in S$}\big\}.
\end{align*}
Note that $R_1^{(\gamma)}$ is non-increasing with $\gamma$. Also, when $G$ is vertex transitive, we have $R_1=R_1^{(\gamma)}+1$ for all $\gamma\in(0,1/2]$, because
in this case, for any given $r$, the size of $B_u^r$ is the same for all $u\in V$.
The theorem below establishes that, for any bounded-degree graph, if there exists a $\gamma\in(0,1/2]$ for which $R_1^{(\gamma)}=\Theta(R_1)$, then
the maximum load when $m=n$ is $\Theta(R_1)$.

\begin{thm}[General version of Theorem~\ref{thm:maxload}]\label{thm:maxload2}
   Let $G$ be any graph with bounded degrees.
   For any constants $\gamma \in (0,1/2]$ and $\alpha\geq 1$, we have
   $$
      \Pro{X_{\max}^{(n)} < \frac{\gamma R_1^{(\gamma)}}{4} } \leq n^{-\omega(1)}
      \quad\text{ and }\quad
      \Pro{X_{\max}^{(n)} \geq 56\alpha R_1} \leq 5n^{-\alpha}.
   $$
\end{thm}
\begin{proof}
   We start establishing a lower bound for $X_{\max}^{(n)}$.
   Let $A$ be a Poisson random variable with mean $1$.
   We first consider the Poissonized versions of the local search allocation and the $1$-choice process (recall the definition of these variants from the paragraph preceding Lemma~\ref{lem:poisson}).
   For any $v\in V$ and any $\ell>0$, Lemma~\ref{lem:pyramid} gives that
   $$
      \Pro{X_v^{(n)} \geq \ell}
      \geq \prod_{r=0}^{\ell-1}\left(\Pro{A \geq \ell-r}\right)^{|N_v^r|}
      \geq \prod_{r=0}^{\ell-1}\left(e^{-1} (\ell-r)^{-\ell+r}\right)^{|N_v^r|},
   $$
   where $N_v^r$ is the set of vertices at distance $r$ from $v$. Since $B_v^\ell=\bigcup_{r=0}^\ell N_v^r$,
   $$
      \Pro{X_v^{(n)} \geq \ell}
      \geq \exp\left(-|B_v^{\ell}| -\ell|B_v^{\ell}|\log(\ell)\right)
      \geq \exp\left(-2\ell|B_v^{\ell}|\log(\ell)\right),
   $$
   where the last step follows for all $\ell\geq2$.
   Given $\gamma>0$, set $\ell=\frac{\gamma R_1^{(\gamma)}}{4}$. Hence, since $|B_v^{r}|\log r$ is increasing with $r$, we have that there exists a
   set $S$ with $|S|= \lceil n^{\frac{1}{2}+\gamma}\rceil$ such that
   \begin{equation}
      \Pro{X_v^{(n)} \geq \frac{\gamma R_1^{(\gamma)}}{4}}
      \geq \exp\left(-\frac{\gamma R_1^{(\gamma)}|B_v^{R_1^{(\gamma)}}|\log(R_1^{(\gamma)})}{2}\right)
      \geq n^{-\gamma/2}
      \quad\text{ for all $v\in S$}.
      \label{eq:sets}
   \end{equation}
   Let $Y=Y(\gamma)$ be the random variable defined as the number of vertices $v$ satisfying $X_{v}^{(n)} \geq \frac{\gamma R_1^{(\gamma)}}{4}$.
   Let $K$ be the total number of balls allocated in the Poissonized version of the local search allocation. Note that $\Ex{K}=n$ and by the last statement of \lemref{AlonSpencer}, $\Pro{ K > 2en } \leq 2^{1-2ne}$.
   Regard $Y$ as a function of the $K$ independently chosen birthplaces
   $U_1,U_2,\ldots,U_K$. Then, for any given $K$, $Y$ is $1$-Lipschitz by \lemref{lipschitz}, and~\eqref{eq:sets} implies that
   $$
      \Ex{Y \mid K\leq 2\ce n}
      \geq n^{\frac{1}{2}+\gamma} \cdot \left(\frac{n^{-\gamma/2}-\Pro{K>2\ce n}}{\Pro{K\leq 2\ce n}}\right)
      \geq \frac{n^{\frac{1}{2}+\frac{\gamma}{2}}}{2}.
   $$
   With this, we apply \lemref{mobd} to obtain
   \begin{align*}
     &\Pro{ X_{\max}^{(n)} < \frac{\gamma R_1^{(\gamma)}}{4}}\\
     &\leq \Pro{ |Y - \Ex{Y \mid K\leq 2\ce n}| \geq \frac{1}{2} \Ex{Y \mid K\leq 2\ce n} \, \Big| \, K \leq 2\ce n } + \Pro{ K > 2 \ce n} \\
     &\leq n^{-\omega(1)} + 2^{1-2n\ce} = n^{-\omega(1)}.
   \end{align*}
   This result can then be translated to the non-Poissonized model via Lemma~\ref{lem:poisson}.

   Now we establish the upper bound, where we consider the non-Poissonized process. For any fixed $u\in V$, we have from the second part of Proposition~\ref{pro:upperbound} (with $m=n$) that
   \begin{align*}
      \Pro{X_u^{(n)} \geq 56 \alpha R_1} &\leq 4\exp\left(-\frac{28\alpha R_1 |B_u^{28 \alpha R_1}| }{14} \log (28\alpha R_1)\right) + \exp\left(- \frac{n}{4} \right) \\
      &\leq 4\exp\left(-2 \alpha  R_1 |B_u^{R_1}|\log R_1\right) + \exp\left(- \frac{n}{4} \right)
      \leq 5 n^{-2 \alpha }.
   \end{align*}
   Taking the union bound over $u$ we obtain that
   $$
     \Pro{X_{\max}^{(n)} \geq 56\alpha R_1} \leq 5 n^{-2 \alpha+1}  \leq 5 n^{-\alpha}.
   $$
\end{proof}

\begin{proof}[{\bf Proof of Theorem~\ref{thm:highdensity}}]
   Applying Proposition~\ref{pro:upperbound} with $\ell=\left(\frac{m}{n}+R_2\right)c$ for any constant $c\geq300\Delta$, we obtain
   \begin{align*}
      &\Pro{\sum_{u\in B_u^{R_2}}X_u^{(m)} \geq \left(\frac{m}{n}+R_2\right)c \cdot |B_u^{R_2}|}\\
      &\leq 4\exp\left(-\left(\frac{m}{n}+R_2\right)\frac{c|B_u^{R_2}|}{14} \log c\right)+\exp\left(-\frac{m}{4}\right)\\
      &\leq 4\exp\left(-\frac{cR_2|B_u^{R_2}|}{14} \log c\right)+\exp\left(-\frac{m}{4}\right).
   \end{align*}
   By setting $c \geq 1$ sufficiently large, the right-hand side above can be made smaller than $n^{-2}$.
   If $u$ has load $k$, then the number of balls allocated to vertices in $B_u^{R_2}$ is at least
   $$
      \sum_{i=0}^{R_2} (k-i)|N_u^i| \geq (k-R_2)|B_u^{R_2}|.
   $$
   Therefore we obtain that, on the event $\sum_{u\in B_u^{R_2}}X_u^{(m)} \leq \left(\frac{m}{n}+R_2\right)c|B_u^{R_2}|$, we have
   $X_u^{(m)} \leq c\left(\frac{m}{n}+R_2\right)+R_2\leq 2c \left(\frac{m}{n}+R_2\right)$. Taking the union bound over all $u$ completes the proof.
\end{proof}

\section{Cover time}\label{sec:covertime}

The proposition below gives an upper bound for the cover time.
\begin{pro}\label{pro:blanket}
   Let $G$ be a graph with bounded degrees. Then for any $\alpha>1$ there exists a $C=C(\alpha,\Delta)>0$ such that for all
   $m\geq C R_2 n$ we have
   $$
      \Pro{X_{\min}^{(m)} < \frac{m}{224 n \log \Delta}} \leq n^{-\alpha},
   $$
   where $X_{\min}^{(m)}=\min_{v\in V} X_v^{(m)}$.
\end{pro}
\begin{proof}
   Fix an arbitrary vertex $u \in V$.
   We will use the concept of weights defined in Section~\ref{sec:key}.
   Define $\mu(v)=d_G(u,v)$ and $W_v^{(m)}=X_v^{(m)}+\mu(v)$.
   Similarly, for the $1$-choice process, define $\overline{W}_v^{(m)}=\overline{X}_v^{(m)}+\mu(v)$.
   Let $Y:= \min_{v \in V} \overline{W}_v^{(m)}$
   be the minimum weight of all vertices in $V$ in the $1$-choice process.
   Let $\ell=\frac{m}{28n \log \Delta}$.
   We have
   \begin{align*}
     \Pro{ Y < \ell}
     &= \Pro{\bigcup_{v \in B_u^{\ell-1}} \left \{ \overline{W}_v^{(m)} < \ell \right\}}\\
     &\leq |B_u^{\ell}|\Pro{ \overline{X}_u^{(m)} < \ell} \\
     &\leq |B_u^{\ell}|\Pro{\Big| \overline{X}_u^{(m)}-\Ex{\overline{X}_u^{(m)}}\Big| > \frac{m}{n}\left(1-\frac{1}{28\log\Delta}\right)}.
   \end{align*}
    Using Lemma~\ref{lem:improvedhoeffding}, we obtain
   $$
     \Pro{ Y < \ell}
     \leq |B_u^{\ell}| \exp \left( - \frac{\frac{m^2}{n^2}\left(1-\frac{1}{28\log\Delta}\right)^2}{\frac{7m}{3n}} \right)
     \leq |B_u^{\ell}| \exp \left( - \frac{3m}{28n}\right)
     \leq \exp \left(\frac{m}{28 n} - \frac{3m}{28n} \right) \leq \frac{1}{2},
   $$
   where the last inequality holds since $m/n \geq C R_2 = \omega(1)$ for bounded degree graphs.
   Now define $\overline{Z}$ as the sum of the $|B_u^{R_2}|$ smallest values of $\left\{\overline{W}_v^{(m)} \colon v \in V \right\}$ and $Z$ as the sum of the $|B_u^{R_2}|$ smallest values of $\left\{ W_v^{(m)} \colon v \in V \right\}$. By \lemref{majorization}, we can couple $W^{(m)}$ and $\overline{W}^{(m)}$ so that, with probability $1$, $Z \geq \overline{Z}$. Further,
   \begin{align*}
     \Ex{ \overline{Z} } &\geq \frac{\ell |B_u^{R_2}|}{2}.
   \end{align*}
   We now apply \lemref{azumavar} in order to show that $\overline{Z}$ is likely to be at least $\frac{\ell |B_u^{R_2}|}{4}$.
   Let $A_1,A_2,\ldots,A_{m}$ be the martingale adapted to the filtration
   $\mathcal{F}_i$ generated
   by $U_1,U_2,\ldots,U_i$; i.e., $A_i = \Ex{\overline{Z} \,\mid\, \mathcal{F}_i}$.
   Since changing the birthplace of ball $i$ (and keeping all other birthplaces the same) can change $Z$ by at most one,
   we have that
   \[
     \Ex{ |A_{i} - A_{i-1}| \, \big\vert \, \mathcal{F}_{i-1} } \leq 1.
   \]
   Now fix $i\in\{1,2,\ldots,m\}$.
   Let $\zeta_u$ be the value of $A_i$ when $U_i=u$ and let
   $\bar \zeta = \frac{1}{n}\sum_{u\in V}\zeta_u$.
   Then we have
   \begin{align*}
      \mathbf{E}_{U_i}\left[(A_i - A_{i-1})^2 \Mid \bigcap\nolimits_{j=1}^{i-1} \{U_j=u_j\}\right] = \frac{1}{n} \sum_{u \in V} (\zeta_u-\bar\zeta)^2,
   \end{align*}
   where the expectation above is taken with respect to $U_i$.
   Since $|\zeta_u-\zeta_{u'}|\leq 1$ for all $u,u' \in V$,
   we can write
   \begin{align*}
      \frac{1}{n} \sum_{u \in V} (\zeta_u-\bar\zeta)^2
      \leq \frac{1}{n} \sum_{u\in V} |\zeta_u-\bar\zeta|
      = \frac{1}{n} \sum_{u\in V}\bigg|\sum_{u'\in V}\frac{1}{n}(\zeta_{u} - \zeta_{u'})\bigg|
      \leq \frac{1}{n^2} \sum_{u\in V}\sum_{u'\in V}\left|\zeta_{u} - \zeta_{u'}\right|.
   \end{align*}
   Now consider a given realization of $U_1,U_2,\ldots,U_{i-1},U_{i+1},\ldots,U_m$, and let $\Gamma\subset V$ be the set of $|B_u^{R_2}|$ vertices with
   smallest loads (note that we skip the $i$th ball in the definition of $\Gamma$). Then, by adding the $i$th ball,
   $\zeta_{u}$ and $\zeta_{u'}$ only differ if at least one of $u$ or $u'$ is in $\Gamma$.
   Hence,
   $\sum_{u\in V}\sum_{u'\in V}\left|\zeta_{u} - \zeta_{u'}\right| \leq 2 |B_{u}^{R_2}|n$.
   Consequently,
   \[
     \mathbf{E}_{U_i}\left[(A_i - A_{i-1})^2 \Mid \bigcap\nolimits_{j=1}^{i-1} \{U_j=u_j\}\right] \leq \frac{2|B_u^{R_2}|}{n}.
   \]
   Now, \lemref{azumavar} gives
   \begin{align*}
     \Pro{ \overline{Z} < \frac{\ell |B_u^{R_2}|}{4}}
     \leq \Pro{ |\overline{Z}-\Ex{\overline{Z}}| \geq \frac{1}{2} \Ex{\overline{Z}} }
     &\leq \exp \left(-  \frac{  \left( \frac{1}{2} \Ex{\overline{Z}} \right)^2 }{ 4 \cdot \frac{|B_u^{R_2}|}{n} \cdot m + \frac{1}{6} \Ex{\overline{Z}}   }  \right).
   \end{align*}
   Clearly, $\Ex{\overline{Z}} \leq \frac{m |B_u^{R_2}|}{n}$, which gives that
   \[
     \Pro{ \overline{Z}  < \frac{\ell |B_u^{R_2}|}{4}}
     \leq \exp \left(-  \frac{  \Ex{\overline{Z}}^2 }{ 16 \cdot \frac{|B_u^{R_2}|}{n} \cdot m + \frac{2 m |B_u^{R_2}|}{3 n}    }  \right)
     \leq \exp \left(-  \frac{  \ell^2 |B_u^{R_2}|/4}{ 17 m/n}  \right).
   \]
   Using the value of $\ell$ and $m$, we have
   \[
     \Pro{ \overline{Z} < \frac{\ell |B_u^{R_2}|}{4}}
     \leq \exp \left(-  \frac{  \frac{m}{n} |B_u^{R_2}|}{68(28\log\Delta)^2}  \right)
     \leq \exp \left(-  \frac{  C R_2 |B_u^{R_2}|}{68(28\log\Delta)^2}  \right)
     \leq n^{-\frac{C}{68(28\log\Delta)^2}}.
   \]
   Due to our coupling which gives $Z \geq \overline{Z}$ we conclude that with probability
   at least $1-n^{-\frac{C}{68(28\log\Delta)^2}}$
   there exists a vertex $v \in B_{u}^{R_2}$ with $W_{v}^{(m)} \geq \frac{\ell}{4}$ and thus $X_v^{(m)} \geq
   \frac{\ell}{4} - R_2$. Then, by smoothness of the load vector (cf.\ Lemma~\ref{lem:balanced}), we have that with probability at least $1-n^{-\frac{C}{68(28\log\Delta)^2}}$,
   every vertex in $B_u^{R_2}$ has load at least $\frac{\ell}{4} - 3R_2 \geq \frac{m}{224 n \log \Delta}$, where the last step follows for all $C\geq 672\log\Delta$.
   Then the result follows by taking the union bound over all $u \in V$, which gives that with probability at least $1-n^{-\frac{C}{68(28\log\Delta)^2}+1}$, all vertices have load
   at least $\frac{m}{224 n \log \Delta}$. The proof is then completed by setting $C$ large enough with respect to $\alpha$ so that $\frac{C}{68(28\log\Delta)^2}-1\geq \alpha$.
\end{proof}

We prove a stronger version of Theorem~\ref{thm:covertime}, which holds also for non-transitive graphs.
For $\gamma\in(0,1/2]$, let
\begin{align*}
   R_2^{(\gamma)} = R_2^{(\gamma)}(G) = \max\big\{r\in\mathbb{N} \colon & \text{there exists } S\subseteq V \text{ with } |S|\geq n^{\frac{1}{2}+\gamma} \\
       & \text{ such that } r |B_u^{r}| < \log n \text{ for all $u\in S$}\big\}.
\end{align*}
Note that $R_2^{(\gamma)}$ is non-increasing with $\gamma$.
Also, when $G$ is vertex transitive, we have $R_2=R_2^{(\gamma)}+1$ for all $\gamma \in (0,1/2]$, because
in this case, for any given $r$, the size of $B_u^r$ is the same for all $u\in V$.
The theorem below establishes that, for any bounded-degree graph, if there exists a $\gamma\in(0,1/2]$ for which $R_2^{(\gamma)}=\Theta(R_2)$, then
the cover time is $\Theta(R_2)$.
\begin{thm}[General version of Theorem~\ref{thm:covertime}]\label{thm:covertime2}
   Let $G$ be any graph with bounded degrees.
   For any $\gamma\in(0,1/2]$ and $\alpha\geq 1$, there exists $C=C(\alpha,\Delta)$ so that
   $$
      \Pro{T_\cov < \frac{\gamma R_2^{(\gamma)}n}{8\Delta} } \leq n^{-\omega(1)}
      \quad\text{ and }\quad
      \Pro{T_\cov \geq C R_2 n} \leq n^{-\alpha}.
   $$
\end{thm}
\begin{proof}
   The second inequality is established by Proposition~\ref{pro:blanket}.
   For the first inequality,
   let $S$ be a set of $n^{\frac{1}{2}+\gamma}$ vertices $u$ for which $R_2^{(\gamma)} \cdot |B_u^{R_2^{(\gamma)}}|<\log n$.
   Let $m= \frac{\gamma R_2^{(\gamma)}n}{8 \Delta}$.
   We consider the Poissonized version of the local search allocation and the 1-choice process. We abuse notation slightly and let $X_v^{(m)}$ and $\bar X_v^{(m)}$ denote the load of $v$ for the Poissonized version of the local search allocation and 1-choice process, respectively, when the expected number of balls allocated in total is $m$.
   For any $u\in S$, we will bound the probability that $X_u^{(m)}=0$.
   First note that since Lemma~\ref{lem:pyramid} works for all $m$, it also works for the Poissonized version by conditioning on the total number of balls.
   Applying the second part of Lemma~\ref{lem:pyramid}, we have that
   $$
      \Pro{X_u^{(m)}=0}
      \geq \Pro{\bigcap\nolimits_{w\in V}\Big\{\overline{X}_w^{(m)}\leq d_G(u,w)\Big\}}.
   $$
   Recall that $N_u^r$ is the set of vertices at distance $r$ from $u$ and $B_u^\ell=\bigcup_{r=0}^\ell N_u^r$.
   By independence of the Poissonized model, we can write
   \begin{align*}
      \Pro{X_u^{(m)}=0}
      &\geq \Pro{\bigcap\nolimits_{w\in B_u^{R_2^{(\gamma)}}}\Big\{\overline{X}_w^{(m)}=0\Big\}}\Pro{\bigcap\nolimits_{i>R_2^{(\gamma)}}\bigcap\nolimits_{w\in N_u^i}\Big\{\overline{X}_w^{(m)}\leq i\Big\}}\\
      &\geq \exp\left(-\frac{m|B_u^{R_2^{(\gamma)}}|}{n}\right)\left(1-\sum\nolimits_{i>R_2^{(\gamma)}}\sum\nolimits_{w\in N_u^i}\Pro{\overline{X}_w^{(m)}> i}\right)\\
      &\geq \exp\left(-\frac{m|B_u^{R_2^{(\gamma)}}|}{n}\right)\left(1-2\sum\nolimits_{i>R_2^{(\gamma)}}\sum\nolimits_{w\in N_u^i}\left(\frac{m e}{ni}\right)^i\right),
   \end{align*}
   where the last inequality follows by the last statement of Lemma~\ref{lem:AlonSpencer}. Using the simple bound $|N_u^i| \leq\Delta^i$ and the fact that
   $\frac{m e\Delta}{ni}\leq \frac{1}{2}$ for all $i\geq R_2^{(\gamma)}$ (as $\Delta/R_2^{(\gamma)} = o(1)$ since $\Delta=\Oh(1)$), we have
   $$
      \Pro{X_u^{(m)}=0}
      \geq \exp\left(-\frac{m|B_u^{R_2^{(\gamma)}}|}{n}\right)\left(1-4 \left( \frac{me\Delta}{nR_2^{(\gamma)}} \right)^{R_2^{(\gamma)}} \right)
      \geq n^{-\frac{\gamma}{8\Delta}} \cdot \frac{1}{2}
      \geq n^{-\frac{\gamma}{8}} \cdot \frac{1}{2}.
   $$
   Now let $Y$ be the random variable defined as the number of vertices $v\in S$ satisfying $X_{v}^{(m)}=0$.
   Let $K$ be the random variable for the total number of balls allocated and regard $Y$ as a function of the $K$ independently chosen birthplaces
   $U_1,U_2,\ldots,U_K$. Then, $Y$ is $1$-Lipschitz by \lemref{lipschitz} for any given $K$. The calculations above give that
   $$
      \Ex{Y \mid K\leq 2em} \geq \Ex{Y} \geq \frac{n^{\frac{1}{2}+\frac{7\gamma}{8}}}{2}.
   $$
   Note that $m =\frac{\gamma R_2^{(\gamma)}n}{8 \Delta}=\Oh(n\log n) $ for any $G$.
   With this, we apply \lemref{mobd} and the last statement of Lemma~\ref{lem:AlonSpencer} to obtain
   \begin{align*}
     &\Pro{ X_{\min}^{(n)} = 0}\\
     &\leq \Pro{ \{|Y - \Ex{Y \mid K\leq 2\ce m}| \geq \frac{1}{2} \Ex{Y \mid K\leq 2\ce m}\} \, \Big| \, \{ K\leq 2\ce m\}} + \Pro{K> 2\ce m}\\
     &\leq 2\exp\left(-\frac{n^{1+14\gamma/8}}{8 (2\ce m)}\right) + 2^{1-2m\ce}= n^{-\omega(1)}.
   \end{align*}
   This result can then be translated to the non-Poissonized process using Lemma~\ref{lem:poisson} and the fact that $m=\Oh(n\log n)$.
\end{proof}

We now state and prove a stronger version of Theorem~\ref{thm:dense}.
\begin{thm}[General version of Theorem~\ref{thm:dense}]\label{thm:densefull}
   Let $G$ be any $d$-regular graph.
   Then, for any $\alpha >1$ there exists $C=C(\alpha) > 0$ such that
   \[
     \Pro{ T_\cov \geq C \cdot \left( n \Big(1 + \frac{\log n \cdot \log d}{d} \Big) \right) } \leq n^{-\alpha}.
   \]
\end{thm}
\begin{proof}
   The result is shown by a coupling with the following stochastic process, introduced in \cite{AHKV03}, which we call {\em coupon collector process}. Initially, every node of $G$ is uncovered. Then in each round $i$, a node $\widetilde{U}_i$ is chosen independently and uniformly at random. If node $\widetilde{U}_i$ is uncovered, then it becomes covered. Otherwise, if $\widetilde{U}_i$ has any uncovered neighbor, then a random node among this set becomes covered. For this process, let us denote by $\widetilde{C}^{(i)}$ the set of covered nodes after round $i$. We shall prove that there is a coupling so that for every round $i$, $\widetilde{C}^{(i)} \subseteq \{ v \in V \colon X_v^{(i)} \geq 1 \}$; in other words, every node which is covered by the coupon collector process in round $i$ is also covered by the local search allocation after the allocation of ball $i$.

   The coupling is shown by induction. Clearly, the claim holds for $i=1$. Consider now the execution of any round $i+1$, assuming that the induction hypothesis holds for round $i$. In our coupling, we choose the same node $v$ for $\widetilde{U}_{i+1}$ and $U_{i+1}$.

   In the first case, we assume that $v$ is uncovered in the coupon collector process. Then the coupon collector process will cover node $v$ in round $i+1$. If $v$ has not been covered by the local search allocation, then we have $X_v^{(i)}=0$ and hence ball $i+1$ will be allocated on node $v$ in round $i+1$. Otherwise, $v$ has been covered previously. In either case, we conclude that node $v$ is covered after round $i+1$ in the local search allocation.

   For the second case, suppose that $v$ is covered in the coupon collector process. Then the coupon collector process will try to cover an uncovered neighbor of $v$ if there exists one. This uncovered neighbor is chosen uniformly at random from all uncovered neighbors of $v$. This random experiment can be described by first choosing a random ranking of all $\deg(v)$ neighbors and then picking the uncovered neighbor with the highest rank, say node $u$. In our coupling, we assume that the local search allocation chooses the same ranking of all $\deg(v)$ neighbors. This, together, with the induction hypothesis, guarantees that if there is node $u$ which becomes covered by the coupon collector process, then this node $u$ also becomes covered by the local search allocation if it has not been covered in an earlier round.

   Combining the two cases, we have shown that there is a coupling such that $\widetilde{C}^{(i)} \subseteq \{ v \in V \colon X_v^{(i)} \geq 1 \}$ for any integer $i \geq 1$. Since it was shown for the coupon collector process in \cite{AHKV03} that with probability $1-n^{-c}$ for some constant $c>0$, $O(n(1+\frac{\log n \cdot \log d}{d}))$ rounds suffice to cover all nodes, the theorem follows.
\end{proof}

\section{Remarks and open questions}

\subsection*{Blanket time}
In analogy with the cover time for random walks, for each $\delta > 1$, we can define the blanket time as the first time at which the load of each vertex is in the interval $(\frac{1}{\delta} \cdot \frac{m}{n},\delta \cdot \frac{m}{ n})$. It follows from Theorem~\ref{thm:highdensity} and Proposition~\ref{pro:blanket} that,
for bounded-degree vertex-transitive graphs, the blanket time is $\Theta(n R_2)$ for all large enough $\delta$.

\subsection*{Extreme graphs}
Note that for any connected graph $G$, we have $R_1(G)\leq \sqrt{\frac{\log n}{\log \log n}}$  and
$R_2(G) \leq \sqrt{\log n}$. Thus, the cycle is the graph with the largest possible maximum load (when $m=n$) and largest possible cover time among
all bounded-degree graphs up to constant factors. Also, for any graph $G$ with bounded degrees,
we have $R_1(G)$ and $R_2(G)$ of order $\Omega(\log\log n)$.
Thus bounded-degree expanders are the graphs with the smallest maximum load (when $m=n$) and smallest cover time among all bounded-degree graphs up to constant factors.

\subsection*{Maximum load and birthplace}
Consider the case $m=n$ and let $u_\star$ be the vertex at which the maximum number of balls are born.
A simple lower bound for the maximum load can be obtained by considering only the balls that are born at $u_\star$. Then, by monotonicity (cf.\ Lemma~\ref{lem:monotonicity}), we may consider
the scenario where $\frac{\log n}{\log \log n}$ balls are born at $u_{\star}$ and no other vertex is a birthplace of a ball.
Since a ball born at a vertex with load $\beta$ will be allocated within distance $\beta$, it follows that
\begin{equation}
   X_{\max}^{(n)} \geq \min\Big\{\beta\geq 0 \colon \beta |B_{u_\star}^\beta| \geq \frac{\log n}{\log\log n}\Big\}.
   \label{eq:naivebound}
\end{equation}
 This is the strategy employed in~\cite[Lemma~2.4]{BSSS13}. Below we show that this is not tight even for vertex-transitive graphs; in other words, the balls born at the single vertex $u_\star$ do not determine the maximum load.

Consider the graph $G$ obtained via the Cartesian product of a vertex-transitive expander of size $n_0=\frac{\log n}{(\log \log n)^3}$ and a cycle of size $n/n_0$.
Let $r_0$ be the diameter of the expander; we have $r_0=\Theta(\log\log n)$. Then, for any vertex $u$ of $G$ and any $r \geq 1$, we have
\begin{equation}
   (r-r_0) n_0 \leq |B_u^{r}| \leq (2r+1) n_0.
   \label{eq:sizeballs}
\end{equation}
Therefore, the bound in~\eqref{eq:naivebound} gives that
$X_{\max}^{(n)} = \Omega(\log\log n)$.

Now we estimate the maximum number of balls that are born in each expander graph of $G$. This can be done via a 1-choice process with
$n'=\frac{n}{n_0}$ bins
and $m'=n=\Theta\left(\frac{n' \log n'}{(\log \log n')^3}\right)$ balls, which using~\cite[Theorem~1]{RS98} has maximum load
$$
   \Theta\left(\frac{\log n'}{\log \frac{n'\log n'}{m'}}\right)
   \geq \frac{\epsilon \log n}{\log\log\log n},
$$
for some constant $\epsilon>0$ and all large enough $n$.
Let $S$ denote the expander with the maximum number of balls born there. Then, using the same reasoning as in~\eqref{eq:naivebound}, and denoting by $B_S^\beta$ the set of vertices within distance $\beta$ from $S$, we have
\begin{align*}
   X_{\max}^{(n)}
   &\geq \min\left\{\beta \geq 0\colon \beta |B_S^\beta| \geq \frac{\epsilon \log n}{\log\log\log n}\right\}\\
   &\geq \min\left\{\beta \geq 0 \colon \beta(2\beta+1) \frac{\log n}{(\log\log n)^3}  \geq \frac{\epsilon \log n}{\log\log\log n}\right\}\\
   &= \Omega\left(\frac{(\log \log n)^{1.5}}{(\log\log\log n)^{0.5}}\right),
\end{align*}
where the first inequality follows from~\eqref{eq:sizeballs}.

\subsection*{Open questions}
\begin{enumerate}
   \item For any vertex-transitive graph (not necessarily of bounded degrees), does it hold that $X_{\max}^{(n)}=\Theta(R_1)$ and $T_\cov=\Theta(R_2 n)$ with high probability?
   \item For any vertex-transitive graph (not necessarily of bounded degrees) and any $m=\omega(n R_2)$, does it hold that $X_{\max}^{(m)}=\frac{m}{n}+\Theta(R_2)$ with high probability?
   \item For any vertex-transitive graph, is the blanket time of order $nR_2$ for all $\epsilon\in(0,1)$? In particular, is the blanket time of the same order as the cover time for all vertex-transitive graphs?
   \item Let $G=(V,E)$ and $G'=(V,E')$ be two graphs such that $E\subset E'$. For any $m$, does it hold that
   the maximum load on $G'$ is stochastically dominated by the maximum load on $G$?
\end{enumerate}
\bibliographystyle{abbrv}
\bibliography{ballbin}

\appendix
\section{Standard technical results}
\begin{lem}[{\cite[Lemma~1.2]{McD89}}]\label{lem:mobd}
Let $X_1,X_2,\ldots,X_n$ be independent random variables with $X_k$ taking values in a set $\Lambda_k$ for each $k$. Suppose that the measurable function $f: \prod_{k=1}^n \Lambda_k \rightarrow \mathbb{R}$ satisfies for every $k$ that
\[
   | f(x) - f(x') | \leq c_k,
\]
whenever the vectors $x$ and $x'$ differ only in the $k$th coordinate. Then for any $\lambda > 0$,
\[
  \Pro{ | f - \Ex{f} | \geq \lambda} \leq 2 \cdot \exp\left(-
  \frac{2 \lambda^2}{ \sum_{k=1}^{n} c_k^2  }  \right).
\]
\end{lem}

\begin{lem}[{\cite[Theorem~6.1]{CL06}}]\label{lem:azumavar}
   Let $X_0,X_1,\ldots,X_m$ be a martingale adapted to the filtration $\mathcal{F}_i$.
   Suppose that there exists a fixed positive $c$ for which $|X_i-X_{i-1}|\leq c$ for all $i$ and
   there exists $c'$ such that $\Ex{(X_i-X_{i-1})^2 \Mid \mathcal{F}_{i-1}}\leq c'$ for all $i$.
   Then,
   $$
      \Pro{|X_m - X_0| \geq \lambda} \leq \exp\left(-\frac{\lambda^2}{2c'm + c\lambda /3}\right).
   $$
\end{lem}

For the special case where $X_1,\ldots,X_m$ are independent Bernoulli random variables, we can apply the above lemma
to the random variables $(X_i-\Ex{X_i})_i$ with $c'=\Ex{X_1}$ and $c=1$ to obtain the inequality below.
\begin{lem}\label{lem:improvedhoeffding}
   Let $X_1,\ldots,X_m$ be $m$ independent, identically distributed Bernoulli random variables.
   Let $X:=\sum_{i=1}^m X_i$. Then, for any $\lambda > 0$,
\begin{align*}
 \Pro{ | X - \Ex{X} | \geq \lambda } \leq \exp \left(- \frac{\lambda^2}{ 2\Ex{X} + \lambda/3 }		 \right).
\end{align*}
\end{lem}

\begin{lem}[{\cite[Theorem~A.1.15]{AlonSpencer}}]\label{lem:AlonSpencer}
Let $X$ have Poisson distribution. Then for any $0 < \epsilon < 1$,
\begin{align*}
 \Pro{ X \leq (1- \epsilon) \Ex{X} } &\leq \exp \left(- \frac{\epsilon^2 \Ex{X}}{2}  \right).
\end{align*}
Also, for any $x\geq 2e\Ex{X}$, it follows by Stirling's approximation that
\begin{align*}
 \Pro{ X \geq x } &\leq 2 \left(\frac{\Ex{X} e}{x}\right)^x.
\end{align*}
\end{lem}
\end{document}